\documentclass{amsart}
\usepackage{graphicx} 
\usepackage[active]{srcltx}
\date{July 24, 2017}
\usepackage{verbatim,enumerate}
\title[Surfaes with light-like points]{%
  Surfaces with light-like points
  in Lorentz-Minkowski 3-space with applications
}

\author[M. Umehara]{Masaaki Umehara}
\author[K. Yamada]{Kotaro Yamada}

\address[Umehara]{%
   Department of Mathematical and Computing Sciences,
   Tokyo Institute of Technology,
   Tokyo 152-8552, Japan
}
\email{umehara@is.titech.ac.jp}

\address[Yamada]{%
   Department of Mathematics,
   Tokyo Institute of Technology,
   Tokyo 152-8551, Japan
}
\email{kotaro@math.titech.ac.jp}

\subjclass[2010]{%
 Primary 53A10;   
 Secondary 53B30, 
           35M10. 
}%
\keywords{%
    maximal surface, 
    mean curvature, 
    type change, 
    zero mean curvature, Lorentz-Minkowski space}%
\thanks{
Umehara was partially supported by 
the Grant-in-Aid for Scientific Research  (A) No.\ 26247005,  
and Yamada by (C) No.\ 26400066 from Japan Society for the 
Promotion of Science.
}
\usepackage{amsthm}
\theoremstyle{plain}
 \newtheorem{theorem}{Theorem}[section]
 \newtheorem{proposition}[theorem]{Proposition}

 \newtheorem{lemma}[theorem]{Lemma}
 \newtheorem{corollary}[theorem]{Corollary}
 
\theoremstyle{definition}
 \newtheorem{definition}[theorem]{Definition}
\theoremstyle{remark}
 \newtheorem{remark}[theorem]{Remark}
 \newtheorem*{remark*}{Remark}
 \newtheorem{example}[theorem]{Example}
 \newtheorem*{acknowledgements}{Acknowledgments}
 
\numberwithin{equation}{section}
\renewcommand{\theenumi}{{\rm(\arabic{enumi})}}
\renewcommand{\labelenumi}{\theenumi}

\newcommand{\op}[1]{{\operatorname{#1}}}
\newcommand{\mb}[1]{\vect{#1}}

\newcommand{\vect}[1]{\boldsymbol{#1}}
\newcommand{\R}{\boldsymbol{R}}

\newcommand{\CC}{\boldsymbol{C}}

\newcommand{\Y}{\mathcal{Y}}
\newcommand{\I}{\mathcal{I}}
\newcommand{\Z}{\mathcal{Z}}
\renewcommand{\phi}{\varphi}
\newcommand{\inner}[2]{\left\langle{#1},{#2}\right\rangle}

\newcommand{\II}{I\!I}
\newcommand{\III}{I\!I\!I}
\newcommand{\sech}{\operatorname{sech}}
\newcommand{\csch}{\operatorname{csch}}
\newcommand{\arccosh}{\operatorname{arccosh}}
\begin{document}
\maketitle
\begin{abstract}
 With several concrete examples of zero mean 
 curvature surfaces in $\R^3_1$ containing a 
 light-like line recently having been found, here
 we construct all real analytic germs
 of zero mean curvature surfaces by
 applying the Cauchy-Kovalevski 
 theorem for partial differential equations.
 A  point where the first fundamental form of 
 a surface degenerates is said to be \emph{light-like}.
 We also show
 a theorem on a property of light-like points
 of a surface in $\R^3_1$ 
 whose mean curvature 
 vector is smoothly extendable.
 This explains why 
 such surfaces will contain a light-like line 
 when they do not change causal types.
 Moreover, several applications of these two results
 are given. 
\end{abstract}

\section*{Introduction}
\label{sec:1}
In this paper, we denote by $\R^3_1$ the Lorentz-Minkowski $3$-space
with inner product $\inner{~}{~}$ with signature $(-++)$,
and write the canonical coordinate system of $\R^3_1$ as 
$(t,x,y)$.

Klyachin \cite{Kl} showed that
a zero mean curvature $C^3$-immersion $F\colon{}U\to\R^3_1$
of a domain $U\subset \R^2$
in the Lorentz-Minkowski $3$-space $\R^3_1$ containing a light-like
point $o$
satisfies one of the following two conditions:
\begingroup
\renewcommand{\theenumi}{(\alph{enumi})}
\renewcommand{\labelenumi}{(\alph{enumi})}
\begin{enumerate} 
 \item\label{enum:klyachin-1} {\it
     There exists a null curve $\sigma$
     {\rm(}i.e.\ a regular curve in $\R^3_1$ 
     whose velocity vector field is light-like{\rm)}
     on the image of $F$
     passing through $F(o)$
     which is non-degenerate {\rm(}i.e.\ its projection
     into the $xy$-plane
     is a locally convex plane curve, 
     cf.\ Definition~\ref{def:non-deg-null}\,{\rm)}.
     Moreover, the causal type of the surface
     changes  from time-like to space-like
     across the curve.}
\item\label{enum:klyachin-2}{\it 
     There exists a light-like line segment 
     passing through $F(o)$
     consisting of the light-like points of $F$.
     Zero mean curvature surfaces which change type 
     across a light-like line belong to this class.}
\end{enumerate} 
\endgroup
The case \ref{enum:klyachin-1} is now well understood
(cf.\ \cite{G}, \cite{Kl} and \cite{KKSY}).
In fact, under the assumption that $F$ is real analytic,
the surface in the class \ref{enum:klyachin-1}
can be reconstructed from the null curve $\sigma$
as follows:
\begin{equation}\label{eq:Bj}
  F(u,v):=
   \begin{cases}
    \dfrac{\sigma(u+i\sqrt{v})+
    \sigma(u-i\sqrt{v})}2 & (v\ge 0), \\[6pt]
    \dfrac{\sigma(u+\sqrt{|v|})+
    \sigma(u-\sqrt{|v|})}2 & (v< 0), 
   \end{cases}
\end{equation}
where $i=\sqrt{-1}$, and we extend the real analytic curve $\sigma$
as a complex analytic map into $\CC^3$.

We call the point $o$ as in the case \ref{enum:klyachin-1}
a \emph{non-degenerate light-like point} of $F$.
A typical example of 
such a surface is obtained by a null curve 
$\gamma(u)=(u, \cos u,\sin u)$ and
the resulting surface is a \emph{helicoid},
which is a zero-mean curvature 
surface  (i.e.\ ZMC-surface)
in $\R^3_1$ as well as in the Euclidean 3-space.
The reference \cite{fluid} is an expository article
of this subject. Moreover, an interesting connection between
type change of ZMC-surfaces
and $2$-dimensional fluid mechanics 
was also given in \cite{fluid}. 
The existence and properties
of entire ZMC-graphs in $\R^3_1$
with non-degenerate light-like points
are discussed in \cite{maxfold}. 
Embedded ZMC-surfaces with 
non-degenerate light-like points with many symmetries are
given in \cite{FRUYY} and \cite{JM}.
 
On the other hand,
several important ZMC-surfaces 
satisfying \ref{enum:klyachin-2} 
are  given in \cite{CR} and \cite{A}.
In contrast to the case \ref{enum:klyachin-1}, 
these examples of surfaces do not change causal 
types across the light-like line%
\footnote{
  In this paper, we say that \emph{a surface changes 
  its causal types across the light-like line}
  if the causal type of one-side of the line
  is space-like and the other-side is time-like.
  If the causal type of the both sides of the line
  coincides, we say that the surface \emph{does not change
  its causal type across the light-like line}.
}.
A family of surfaces constructed in \cite{OJM} 
satisfies \ref{enum:klyachin-2} and also changes its causal type.
In spite of this progress, there was 
still no machinery available to find surfaces
of type \ref{enum:klyachin-2} and no simple explanation for why only two
cases occur at light-like points. 

In this paper, we clarify such phenomena as follows:
We denote by  $\Y^r$ ($r\ge 3$)
the set of germs of
$C^r$-differentiable immersions in $\R^3_1$ 
whose mean curvature vector field can be
smoothly extended at a light-like point.
We prove a property
of regular surfaces in the  class $\Y^3$,
which contains the above 
Klyachin's result as a special case. Our approach
is different from that of \cite{Kl}:
we use the uniqueness of ordinary differential equations
to prove the assertion. 
We also show a general existence of real analytic
ZMC-surfaces and surfaces in $\Y^\omega$
using the Cauchy-Kovalevski theorem for 
partial differential equations.
As a consequence, new examples of 
ZMC-surfaces which 
change type along a given light-like line
are obtained.

\section{Preliminaries}
We denote by $\mb 0:=(0,0,0)$ (resp.\ $o:=(0,0)$) the origin 
of the Lorentz-Minkowski $3$-space
$\R^3_1$ of signature ($-++$)
(resp.\ the plane $\R^2$)
and denote by $(t,x,y)$ the canonical coordinate system of
$\R^3_1$.
An immersion $F:U\to\R^3_1$ of a domain $U\subset\R^2$ into 
$\R^3_1$ is said to be \emph{space-like} (resp.\ \emph{time-like},
\emph{light-like}) at $p$ if the tangent plane of the image $F(U)$
at $F(p)$ is space-like (resp.\ time-like, light-like),
that is, the restriction of the metric $\inner{~}{~}$
to the tangent plane is positive definite (resp.\ indefinite,
degenerate).
We denote by $\widetilde{\I}^r$ ($r\geq 2$)
the set of germs of 
$C^r$-immersions into $\R^3_1$ which map 
the origin $o$ in the $uv$-plane to the origin
$\vect{0}$ in $\R^3_1$.
($F\in \widetilde{\I}^\omega$ means that $F$ is real analytic.)
Let $F:(U,o)\to \R^3_1$
be an immersion in the class $\widetilde{\I}^r$.
We denote by
$U_+$ (resp.\ $U_-$)
the set of space-like (resp.\ time-like) points, 
and set
\[
   U_*:=U_+\cup U_-.
\]
A point $p\in U$ is light-like if $p\not\in U_*$.
We denote by $\widetilde{\I}^{r}_L(\subset \widetilde{\I}^r)$
the set of germs of $C^r$-immersion such that $o$ 
is a light-like point.

If $F\in\widetilde{\I}^r_L$, the tangent plane of the
image of $F$ at $o$ contains a light-like vector
and does not contain time-like vectors.
Thus, we can express the surface as a graph
\begin{equation}\label{eq:F-graph}
 F=(f(x,y),x,y),
\end{equation}
where $f(x,y)$ is a $C^r$-function 
defined on a certain neighborhood of the origin
of the $xy$-plane.
Let 
\begin{equation}\label{eq:B}
    B_F:= 1-f_x^2-f_y^2\qquad 
         \left(f_x = \frac{\partial f}{\partial x}, \quad
                              f_y = \frac{\partial f}{\partial y}\right).
\end{equation}
Then the point of the graph \eqref{eq:F-graph} is space-like (resp.\
time-like) if and only if $B_F>0$ (resp.\ $B_F<0$) at the point.
Since $F\in\widetilde\I^r_L$, 
the origin $o=(0,0)$ is light-like, that is, $B_F(0,0)=0$.
Hence there exists $\theta\in[0,2\pi)$ such that
\[ 
   f_x(0,0)=\sin\theta,\qquad f_y(0,0)=\cos\theta.
\]
So by a rotation about the $t$-axis, we may assume
\begin{equation}\label{eq:normalize}
  f_x(0,0)=0,\qquad f_y(0,0)=1
\end{equation}
without loss of generality.

We denote by $\I^r_L(\subset {\widetilde \I}^r_L)$ 
the set of  germs of $C^r$-immersion
$F$ with properties \eqref{eq:F-graph} and \eqref{eq:normalize}.
Then
\begin{equation}\label{eq:iota}
 \iota:\I^r_L\ni F \mapsto f(=\iota_F)\in 
  \{f\in C^{r}_o(\R^2)\,;\, f(0,0)=f_x(0,0)=0, f_y(0,0)=1\},
\end{equation}
which maps $F$ to the function $f$ as in \eqref{eq:F-graph},
is a bijection, where $C^{r}_o(\R^2)$ is the set of $C^r$-function
germs on a neighborhood of $o\in\R^2$.

For $F\in\I^r_L$, 
\begin{equation}\label{eq:U-pm}
  U_+=\{p\in U\,;\,B_F(p)> 0\},\qquad
  U_-=\{p\in U\,;\,B_F(p)< 0\}
\end{equation}
hold, where $B_F$ is the function as in \eqref{eq:B}. 
We let
\begin{equation}\label{eq:AB}
  A_F:=(1-f_x^2)f_{yy}+2f_xf_yf_{xy}+(1-f_y^2)f_{xx}.
\end{equation}
Then the mean curvature function of $F$ 
\begin{equation}\label{eq:H}
       H_F:=\frac{A_F}{2|B_F|^{3/2}}
\end{equation}
is defined on $U_*$ (cf.\ {\cite[Lemma 2.1]{HKKUY}}).
We first remark the following:
\begin{proposition}[{cf.\ Klyachin \cite[Example 4]{Kl}}]
 If $B_F$ vanishes identically, then so does $A_F$.
\end{proposition}
\begin{proof}
 Since $B_F=0$, we have
 $1-f_x^2=f_y^2$.
 By differentiating this, we get
 $f_x f_{xy}=-f_y f_{yy}$,
 and
 \begin{equation}\label{eq:1}
  (1-f_x^2)f_{yy}=f_y(f_yf_{yy})=-f_x f_yf_{xy}.
 \end{equation}
 Similarly, we have
 \begin{equation}\label{eq:2}
  (1-f_y^2)f_{xx}=f_x^2f_{xx}=-f_x f_yf_{xy}.
 \end{equation}
 By \eqref{eq:1} and \eqref{eq:2}, we get the
 identity $A_F=0$.
\end{proof}

We denote by $\Lambda^r$ the set of germs 
of immersions $F\in \I^r_L$ with identically vanishing $B_F$,
that is, $\Lambda^r$ is the set of germs of 
\emph{light-like immersions.}
We denote by $C^\omega_{o}(\R,0_2)$ the set of 
real analytic functions $\phi$ satisfying 
$\phi(0)=d\phi(0)/dx=0$.
Then the following assertion holds:

\begin{proposition}\label{prop:CK0}
 The map
 \[
   \lambda:\Lambda^\omega\ni F \mapsto (\lambda_{F}:=)f(x,0)
          \in C^\omega_{o}(\R,0_2)
 \]
 is bijective, where $f=\iota_F$ {\rm(}cf.\ \eqref{eq:iota}{\rm)}.
\end{proposition}

\begin{proof}
 Since $B_F$ vanishes identically, taking in account of
 \eqref{eq:normalize},
 we can write 
 \[
   f_{y}=\sqrt{1-f_x^2}.
 \]
 This can be considered as a normal form of 
 a partial differential equation
 under the initial condition
 \begin{equation}\label{eq:psi}
  f(x,0):=\psi(x)\qquad (\psi\in C^\omega_{o}(\R,0_2)),
 \end{equation}
 because of the condition $f_x(0,0)=\psi'(0)=0$.
 So we can apply the Cauchy-Kovalevski theorem (cf.~\cite{KP})
 and show the uniqueness and  existence of the solution
 $f$ satisfying \eqref{eq:psi}.
\end{proof}

\begin{remark}
 The above proof of the existence of a light-like surface
 $F(x,y)$ satisfying \eqref{eq:psi} is local, that is,
 it is defined only for small $|y|$.
 Later, we will show that $F$ is a ruled surface
 and has an explicit expression, 
 see 
 Corollary~\ref{cor:null-ruled} and \eqref{eq:LF}.
\end{remark}

\begin{example}\label{ex:A}
 The light-like plane
 $F(x,y)=(y,x,y)$
 belongs to the class $\Lambda^\omega$
 such that $\lambda_F=0$.
\end{example}

\begin{example}\label{ex:B}
 The light-cone
 $F(x,y)=(\sqrt{x^2+(1+y)^2}-1,x,y)$
 is a light-like surface satisfying
 $\lambda_F=\sqrt{1+x^2}-1$.
\end{example}

\section{Surfaces with smooth mean curvature vector field.}
Let $F:(U,o)\to (\R^3_1,\vect{0})$ be an immersion
of class $\I^r_L$ ($r\ge 3$) such that $U_*$ is 
open and dense in $U$,
and fix a $C^{r-2}$-function  $\phi$.
We say that  $F$ is \emph{$\phi$-admissible} if
\begin{equation}\label{eq:H2}
   A_F-\phi B_F^{2}=0
\end{equation} 
holds,
where $A_F$ and $B_F$ are as in \eqref{eq:AB} and \eqref{eq:B},
respectively.
We denote
\begin{equation}\label{eq:Y}
 \Y^r_\varphi :=\{F\in\I^r_L\,;\,\text{$F$ is $\varphi$-admissible}\}.
\end{equation}
An immersion germ $F\in \I^{r}_L$ is 
called \emph{admissible} if
it is $\phi$-admissible 
for a certain $\phi\in C^{r-2}_o(\R^2)$.
The set
\[
  \Y^r:=\bigcup_{\phi\in C^{r-2}_o(\R^2)}
  \Y^r_{\phi}
\]
consists of all germs of $\phi$-admissible immersions.
The following assertion explains why the class $\Y^r$
is important.

\begin{proposition}[\cite{HKKUY}]
 Let $F:(U,o)\to (\R^3_1,\mb 0)$ 
 be an immersion in the class $\mathcal I^r_L$
 for $r\geq 3$.
 Then the mean curvature vector field $\vect{H}_F$
 can be $C^{r-2}$-differentiablly
 extended on a neighborhood of $o$ 
 if and only if $F$ belongs to the class $\Y^r$. 
\end{proposition}
\begin{proof}
 This assertion follows from the fact that 
 \[
   \vect{H}_F=\frac{A_F}{2(B_F)^2} (F_{x}\times F_{y})
 \] 
 on $U^*$,
 where $\times$ denotes the vector product in $\R^3_1$.
\end{proof}

Since $\phi=0$ implies $\vect{H}_F=0$,  
\begin{equation}\label{eq:Z}
   \Z^r:=\{F\in \I^{r}_L\,;\, A_F=0\}(=\Y^r_0)
\end{equation}
is the set of germs of zero mean curvature immersions
in $\R^3_1$ at the light-like point $o$.
By definition, we have 
\begin{equation}
 \Lambda^r\subset \Z^r \subset \Y^r
  \qquad (r\ge 3).
\end{equation}

Surfaces in the class $\Y^r$ are investigated in \cite{HKKUY}, 
and an entire graph in $\Y^r$ which is not a ZMC-surface was
given. In this section, we shall show a
general existence result of surfaces in the class
$\Y^\omega$. We fix a germ of a real analytic function  
$\phi\in C^\omega_{o}(\R^2)$, and
take an immersion $F\in \Y^\omega_\phi$. 
\begin{definition}
 Let $f:=\iota_F$ be the function associated with
 $F\in \Y^\omega_\phi$ (cf.\ \eqref{eq:iota}).
 Then
 \begin{equation}\label{eq:f0}
  \gamma_F(x):=(f(x,0),f_y(x,0))
 \end{equation}
 is a real analytic plane curve, which we call 
 the \emph{initial curve} 
 associated with $F$.
\end{definition}

We denote by $C_{(0,1)}^\omega(\R,\R^2)$
the set of germs of $C^\omega$-maps $\gamma:(\R,0)\to (\R^2,(0,1))$. 
By definition (cf.\ \eqref{eq:normalize}),
$\gamma_F(x)\in C_{(0,1)}^\omega(\R,\R^2)$
holds. We prove the following assertion:
\begin{theorem}\label{thm:E0}
 For $\phi\in C^\omega(\R^2)$,
 the set $\Y^\omega_\phi$ is non-empty.
 More precisely, the map
 \[
   \Y^\omega_\phi\ni F \mapsto \gamma_F\in C_{(0,1)}^\omega(\R,\R^2)
 \]
 is bijective.
 Moreover, the base point $o$ is a 
 non-degenerate $($resp.~degenerate$)$ light-like point 
 of $F$ if $\dot\gamma_F(0)\ne (0,0)$
 {\rm(}resp.~$\dot\gamma_F(0)=(0,0)${\rm)},
 where ``dot'' denotes $d/dx$.
\end{theorem}

\begin{proof}
 Suppose that $F\in \Y^\omega_\phi$. 
 Since 
 $A_F-\varphi B_F^2$ vanishes identically (cf.\ \eqref{eq:H2}),
 $f=\iota_F$ satisfies
 \begin{equation}\label{eq:pde}
  f_y=g,\quad
  g_{y}=-\frac{2 f_x g g_{x}+(1- g^2)f_{xx}-(1-f_x^2-g^2)^{2}\phi}{1-f_x^2},
 \end{equation}
 which is the normal form for partial differential equations. 
 So we can apply the Cauchy-Kovalevski theorem (cf.\ \cite{KP})
 for a given initial data
 \[
     (f(x,0),g(x,0)):=
    \gamma(x)\qquad (\gamma\in C_{(0,1)}^\omega(\R,\R^2)).
 \]
 Then the solution $(f,g)$ of \eqref{eq:pde} is uniquely 
 determined. Obviously, the resulting immersion 
 $F_\gamma:=(f(x,y),x,y)$ 
 gives a surface in $\Y^\omega_\phi$
 whose initial curve is $\gamma$.
 The second assertion follows from the fact that
 $\dot{\gamma}(0)=(0,0)$ if and only if $\nabla B_F(0,0)=\vect{0}$,
 where $\nabla B_F:=((B_F)_x,(B_F)_y)$.
\end{proof}

When $\phi=0$, we get the following:

\begin{corollary}\label{cor:Z00}
 The map
 $\Z^\omega\ni F \mapsto \gamma_F\in C^\omega_{(0,1)}(\R,\R^2)$
 is bijective. 
\end{corollary}
The following is a direct consequence of
this corollary and Theorem \ref{thm:E0}.
\begin{corollary}
 In the above correspondence, it holds that
 \[
   \Lambda^\omega=\left\{F\in \Z^\omega\,;\,  
    \gamma_F=\left(\psi,\sqrt{1-\dot \psi^2}\right),
     \,\, \psi\in C^\omega_{o}(\R,0_2)
    \right\}.
 \]
\end{corollary}

\section{A property of light-like points.}
\label{sec:3}
\begin{definition}\label{def:non-deg-null}
 Let $I$ be an open interval, and
 $\sigma:I\to \R^3_1$ a  regular curve of class $C^r$ ($r\ge 3$). 
 The space curve $\sigma$ is called \emph{null}
 if $\sigma'(t)=d\sigma/dt$ is light-like.
 Moreover $\sigma$ is called \emph{non-degenerate}
 if $\sigma''(t)$ is not proportional to $\sigma'(t)$ 
 for each $t\in I$.
\end{definition}

The orthogonal projection of a non-degenerate
null curve into the $xy$-plane is a locally
convex plane curve.
The following assertion is a generalization of
Klyachin's result in the introduction,
since ZMC-surfaces are elements of $\Y^3$.

\begin{theorem}
\label{thm:main2}
 Let $F:(U,o)\to \R^3_1$ be an immersion of class $\Y^3$.
 Then, one of the following two cases occurs{\rm:}
 \begingroup
 \renewcommand{\theenumi}{{\rm (\alph{enumi})}}
 \renewcommand{\labelenumi}{\rm (\alph{enumi})}
 \begin{enumerate}
  \item\label{enum:K-1}
       $\nabla B_F$ does not vanish at $o$, and
       the image of the level set $F(\{B_F=0\})$ consists of
       a non-degenerate null regular curve in $\R^3_1$,
       where $f=\iota_F$.
  \item\label{enum:K-2}
       $\nabla B_F$ vanishes at $o$, and
       the image of the level set $F(\{B_F=0\})$ contains
       a light-like line segment in $\R^3_1$ passing through $F(o)$.
 \end{enumerate}
 \endgroup
\end{theorem}
\begin{proof}
 The first assertion \ref{enum:K-1} was proved in
 \cite[Proposition 3.5]{HKKUY}.
 So it is sufficient to prove \ref{enum:K-2}.
 We may assume that $F\in \Y^3_{\phi}$ and
 $\phi\in C^1_o(\R^2)$.
 Let $f:=\iota_F$.
 Since $f$ is of class $C^3$, 
 applying the division lemma (Lemma~\ref{lem:div} in
 Appendix~\ref{app:div}) for $g(x,y):=2(f(x,y)-f(0,y)-xf_x(0,y))$,
 there exists a $C^1$-function $h$
 such that
 \begin{equation}\label{eq:f-rep}
  f(x,y)=a_0(y)+a_1(y)x+\frac{h(x,y)}2x^2
   \quad \bigl(a_0(y):=f(0,y),~a_1(y):=f_x(0,y)\bigr).
 \end{equation}
 By \eqref{eq:normalize},
 it holds that 
 \begin{equation}\label{I1}
   a_0(0)=0,\quad a'_0(0)=1,\qquad
   a_1(0)=0.
 \end{equation}
 Moreover, since $\nabla B_F$ vanishes at $o$, we have
 \begin{equation}\label{I2}
  a'_1(0)=0.
 \end{equation}
 We set (cf.\ \eqref{eq:AB} and
 \eqref{eq:H2}) 
 \[
   \tilde A:=A_F-\phi B_F^{2}.
 \]
 Since $F\in \Y^3_{\varphi}$, 
 $\tilde A$ vanishes identically.
 We set
 \begin{alignat*}{3}
  h_0(y)&:=h(0,y),
  \quad &
  h_1(y)&:=h_x(0,y),
  \quad &
  h_2(y)&:=h_y(0,y), \\
  \phi_0(y)&:=\phi(0,y),
  \quad &
  \phi_1(y)&:=\phi_x(0,y),
  \quad &
 \end{alignat*}
 Then we have
 \begin{align}\label{eq:e1}
  0=\tilde A|_{x=0}=&
  \phi_0 ((a'_0)^2+a_1^2-1)^{2}+h_0 
  \left(1-(a'_0)^2\right)\\
  &\phantom{aaaaaaaaa}+\left(1-a_1^2\right) a''_0+2 a_1 
  a'_0 a'_1, \nonumber \\
  \label{eq:e2}
  0=\tilde A_x|_{x=0}=&
  -2 a_1 h_0 a''_0+4\phi_0 
  \left(1-(a'_0)^2-a_1^2\right) 
  \left(a_1 h_0+a'_0 a'_1\right)\\
  &\phantom{aa}+2 a_1 h_2 a'_0-\phi_1
  \left(-(a'_0)^2-a_1^2+1\right)^{2}
  \nonumber
  \\
  \nonumber
  &\phantom{aaaaa}-h_1 \left((a'_0)^2-1\right) -
  \left(a_1^2-1\right) a''_1
  +2 a_1 (a'_1)^2.
 \end{align}
 These two identities \eqref{eq:e1} 
 and \eqref{eq:e2} can be rewritten in the form
 \begin{align*}
  &\left(1-a_1^2\right) a''_0=
  \Psi_1(x,y,a_0,a'_0,a_1,a'_1), \\
  &-2 a_1 h_0 a''_0
  +\left(a_1^2-1\right) a''_1=
  \Psi_2(x,y,a_0,a'_0,a_1,a'_1),
 \end{align*}
 where  $\Psi_1$ and $\Psi_2$ are continuous functions
 of five variables. 
 Since $1-a_1^2(0)=1$,
 this gives a normal form of 
 a system of
 ordinary differential equations with 
 unknown functions $a_0$ and $a_1$.
 Moreover, this system of differential equations 
 satisfies the local Lipschitz condition, since 
 $\Psi_1$ and $\Psi_2$ are polynomials in
 $a_0$, $a'_0$, $a_1$ and $a'_1$.
 Here,
 \begin{equation}\label{sol}
  (a_0,a_1)=(y,0)
 \end{equation}
 gives a solution of this system of equations. 
 Then the uniqueness of the solution
 with the initial conditions \eqref{I1} and \eqref{I2} implies 
 that \eqref{sol} holds for $F$. As a consequence,
 we have $F(0,y)=(y,0,y)$, proving the assertion. 
\end{proof}

\begin{remark}
 If $F$ is of non-zero constant mean curvature $H$,
 then  $A_F-2H|B_F|^{3/2}$ vanishes
 identically. In this case, we also get
 the relations 
 $A_F|_{x=0}=(A_F)_x|_{x=0}=0$,
 which can be considered as a system of ordinarily
 equations like as in the above proof.
 However, this does not
 satisfy the local Lipschitz condition, and
 the above proof does not work in this 
 case. Fortunately, an analogue of
 Theorem \ref{thm:main2} can be proved
 for surfaces with non-zero constant mean curvature
 using a different approach
 (see \cite{UY} for details).
\end{remark}

As a corollary, we immediately 
get the following:

\begin{corollary}\label{cor:isol}
 Any light-like points on a surface in
 the class $\Y^3$ are not isolated.
\end{corollary}

If a light-like point $o$ is non-degenerate,
then $B_F$ changes sign, that is, $F$ changes causal type.
So the following corollary is also obtained.

\begin{corollary}\label{cor:2}
 If an immersion $F\in\Y^3$
 does not change its causal type at $o$,
 then there exists a light-like line $L$ passing through $f(o)$
 such that the set of the light-like points of $F$
 is a regular curve $\gamma$ and the image of
 $F\circ \gamma$ lies on the line $L$.
\end{corollary}

We next give an application of Theorem \ref{thm:main2}
for light-like surfaces:

\begin{definition}
 Let $\sigma(t)$ ($t\in I$)
 be a space-like $C^\infty$-regular curve defined on a interval $I$.
 Since the orthogonal complement of $\sigma'(t)$ is Lorentzian,
 there exists a non-vanishing vector field $\xi(t)$
 along $\sigma$
 such that $\xi(t)$ points the light-like direction which is 
 orthogonal to $\sigma'(t)=d\sigma(t)/dt$, that is, it holds that
 \[
   \inner{\xi(t)}{\xi(t)}=\inner{\xi(t)}{\sigma'(t)}=0
 \]
 where $\inner{~}{~}$  means the canonical Lorentzian inner product in
 $\R^3_1$. 
 The possibility of such vector fields $\xi(t)$
 are essentially two up to a multiplication of 
 non-vanishing smooth functions. 
 Then the map
 \begin{equation}\label{eq:F}
    F(t,s):=\sigma(t)+s\, \xi(t)\qquad (t\in I, |s|<\epsilon)
 \end{equation}
 gives a light-like immersion if $\epsilon>0$ is 
 sufficiently small (This representation formula
 was given in Izumiya-Sato \cite{IS}).
 We call such a $F$ a 
 \emph{light-like ruled surface} 
 associated to the space-like curve $\sigma$).
\end{definition}

 For example, consider an ellipse
 $\sigma(t)=(0, a \cos t,\sin t)$
 on the $xy$-plane in $\R^3_1$, 
 where $a>0$ is a constant.
 Then an associated light-like surface is
 given by \eqref{eq:F} by setting
 \[
    \xi(t):=\left(\sqrt{a^2 \sin ^2t+\cos^2t},
             \cos t, a \sin t\right).
 \]
 Figure \ref{fig:0} presents the resulting light-like surface
 for $a=2$.
\begin{figure}[htb]
 \centering
          \includegraphics[height=5.0cm]{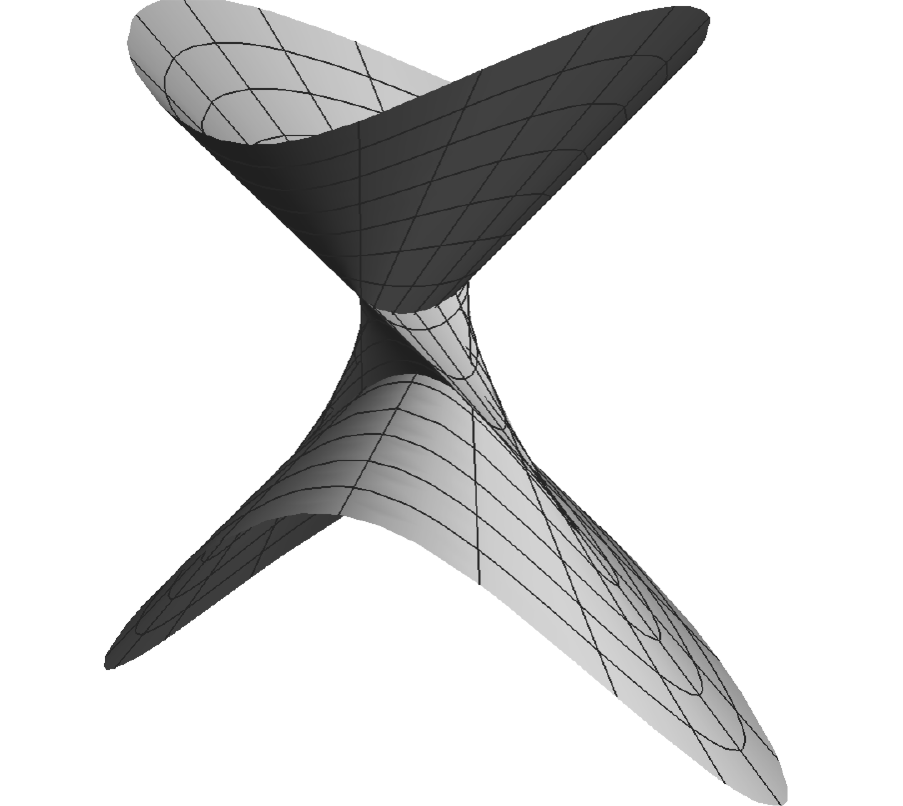}
  \caption{The light-like surface associated to an ellipse.}
  \label{fig:0}
\end{figure}

The following corollary asserts that light-like regular 
surfaces are locally  regarded as ruled surfaces.

\begin{corollary}\label{cor:null-ruled}
 A light-like surface germ $F\in \Lambda^\infty$ 
 can be paranmetrized by a light-like ruled
 surface along a certain space-like regular curve. 
\end{corollary}
\begin{proof}
 Let $F$ be a light-like surface such that
 $\iota_F(x,0)=\psi(x)$ as in \eqref{eq:psi},
 where $\psi(0)=\dot{\psi}(0)=0$ ($\dot{~}=d/dx$).
 Then it holds that $\sigma(x):=F(x,0)=(\psi(x),x,0)$
 is a space-like curve for sufficiently small $x$.
 There are $\pm$-ambiguity of light-like vector fields
 \[
   \xi^\pm(x):=\left(1,\dot\psi(x),\pm \sqrt{1-\dot\psi(x)^2}\right)
 \]
 along the curve $\sigma(x)$ perpendicular to $\dot\sigma(x)$.
 By Corollary \ref{cor:2},
 $F$ must be a ruled surface foliated by light-like lines.
 Since the light-like line in the image of $F$
 passing through the origin is 
 $y\mapsto (y,0,y)$, the light-like ruled surface
 \begin{equation}\label{eq:LF}
  (t,s)\mapsto \sigma(t)+s\, \xi^+(t)
 \end{equation}
 gives a new parametrization of the surface $F$,
 that proves the assertion.
\end{proof}

It should be remarked that the property of 
light-like points (cf.~Theorem \ref{thm:main2})
can be generalized for surfaces in
arbitrarily given Lorentzian 3-manifolds, see \cite{UY}.
It is well-known that space-like ZMC-surfaces have
non-negative Gaussian curvature.
Regarding this fact,
we prove the following assertion on the  sign of
Gaussian curvature at non-degenerate light-like points:

\begin{proposition}\label{prop:PK1}
 Let $F$ be an immersion in the class $\Y^r$ $(r\ge 3)$.
 Suppose that $o$ is a non-degenerate light-like point.
 Then the Gaussian curvature function $K$ diverges 
 to $\infty$ at $o$.
\end{proposition}
When $F$ is of ZMC, the assertion
was proved in Akamine \cite{A2}.

\begin{proof}
 Let $f=\iota_F$ is the function associated with $F\in \Y^r$ 
 (cf.\ \eqref{eq:iota}) with 
 \eqref{eq:f-rep} and  set
 \[
   C_F:=f_{xx}f_{yy}-(f_{xy})^2.
 \]
 Then the Gaussian curvature of $F$ is given by
 \begin{equation}\label{eq:BC}
  K:=-\frac{C_F}{(B_F)^2},
 \end{equation}
 where $B_F$ is the function as in \eqref{eq:B}.
 Since the function $A_F$ as in \eqref{eq:AB}
 satisfies $A_F(o)=a_0''(0)=0$, we have
 \[
    C_F(o)=h(o) a''_0(0)-a'_1(0)^2=-a'_1(0)^2.
 \]
 Here $a'_1(0)\ne 0$ since $o$ is a non-degenerate light-like
 point. Thus $C_{F}(o)\ne 0$ holds, and
 we get the conclusion because of the fact $B_F(o)=0$.
\end{proof}

\section{Properties of surfaces in $\Y^\omega$}

We denote by $\Y^r_{a}$ (resp.\ $\Y^r_b$)
the subset of $\Y^r$ consisting of surfaces
such that the origin $o$ is a non-degenerate
(resp.\ degenerate) light-like point.
Then
\[
   \Y^r:=\Y^r_a\cup \Y^r_{b}
\]
holds.
We next define two subsets of $\Z^r$ (cf.\ \eqref{eq:Z}) as
\begin{align*}
 \Z_{a}^r&:=\{F \in \Z^r\,;\, \gamma_F(0)\ne (0,0)\}
 =\{F \in \Z^r\,;\, \nabla B_F(0,0)\ne \vect{0}\}
  =\Y^r_a\cap \Z^r, \\
 \Z_{b}^r&:=\{F \in \Z^r\,;\, \gamma_F(0)= (0,0)\}
 =\{F \in \Z^r\,;\, \nabla B_F(0,0)=\vect{0}\}
  =\Y^r_b\cap \Z^r,
\end{align*}
where $B_F$ is the function as in \eqref{eq:B}.

As explained in the introduction,
surfaces in  $\Z_{a}^r$ can be constructed using the
formula as in \eqref{eq:Bj}.
On the other hand, to get ZMC-surfaces in $\Z^\omega_b$,
we can apply Corollary \ref{cor:Z00}.
However, it is only a general existence result,
and is not useful if
one would like to know the precise behavior 
of the surfaces along the degenerate light-like lines.
Here, we focus to the class $\Z_b^\omega$.
We first show that $\Y_b^\omega$
and will sow that surfaces in the class $\Y_b^\omega$
have quite similar properties as zero-mean curvature surfaces
in the class $\Z^\omega_b$:
Each surface $F\in \Y^\omega_b$ is expressed as
the graph of a function 
\begin{equation}\label{zm22b}
 f(x,y):=y+\frac{\alpha_F(y)}{2}x^2+\frac{\beta_F(y)}{3}x^3
  +h(x,y)x^4, 
\end{equation}
where 
$\alpha_F$, $\beta_F$ and $h$ are certain real analytic functions.
In fact, as seen in the proof of Theorem~\ref{thm:main2},
$a_1(y) = 0$ holds when $o$ is a degenerate light-like point,
where $a_1(y):=f_x(0,y)$.
We call $\alpha_{F}(y)$ and $\beta_{F}(y)$ 
in \eqref{zm22b}
the \emph{second approximation function} and
the \emph{third approximation function} of $F$,
respectively. 
These functions give the following 
approximation of $F$:
\begin{equation}\label{eq:first}
f(x,y)\approx y+\frac{\alpha_F(y)}{2}x^2+\frac{\beta_F(y)}{3}x^3.
\end{equation}

\begin{proposition}\label{prop:2ode}
 For $F\in \Y_b^\omega$,
 there exists a real number 
 $\mu_F$ {\rm(}called the \emph{characteristic} of $F${\rm)}
 such that $\alpha_F$ and $\beta_F$ satisfy 
 \begin{align}\label{eq:c}
  & \alpha'_F+\alpha_F^2+\mu_F=0, \\ 
    \label{eq:c2}
  & \beta''_F+4 \alpha_F \beta'_F=0.
 \end{align}
 Moreover, if $\mu_F>0$ {\rm(}resp.\ $\mu_F<0${\rm)}
 then $F$ has no time-like points
 {\rm(}resp.\ no space-like points{\rm)}.
 In particular, if $F$ changes causal type, the $\mu_F=0$.
\end{proposition} 

When $F\in \Z_b^\omega$, this 
assertion for $\alpha:=\alpha_F$ was proved in \cite{CR}.
So the above assertion is its generalization. 

\begin{proof}
 Since $F$ in $\Y^\omega$, there exists a function 
 $\varphi\in C^{\omega}_o(\R^2)$ such that $F\in\Y^{\omega}_\phi$.
 Let $f:=\iota_F$ and set $A:=A_F$ and $B:=B_F$
 as in \eqref{eq:B} and \eqref{eq:AB}, respectively.
 Then by \eqref{zm22b} implies that
 \begin{equation}\label{eq:b-zero}
      B(0,y) = B_x(0,y)=0.
 \end{equation}
 Since $\phi={A}/{B^2}$ is a smooth function, the L'Hospital rule
 yields that
 \begin{align*}
  \phi(0,y)
  &=\lim_{x\to 0}\frac{A(x,y)}{B(x,y)^2}
  =\lim_{x\to 0}\frac{A_x(x,y)}{2B_x(x,y)B(x,y)} \\
  &=\lim_{x\to 0}\frac{A_{xx}(x,y)}{2B_{xx}(x,y)B(x,y)+2B_{x}(x,y)B_x(x,y)}\\
  &=\lim_{x\to 0}\frac{A_{xxx}(x,y)}{2B_{xxx}(x,y)B(x,y)+6B_{xx}(x,y)B_x(x,y)}.
 \end{align*}
 Thus, \eqref{eq:b-zero} yields that we have
 \[
        0=A(0,y)=A_x(0,y)=A_{xx}(0,y)=A_{xxx}(0,y).
 \]
 Here
 $A|_{x=0}=A_x|_{x=0}=0$
 does not produce any restrictions for $\alpha:=\alpha_{F}$ and 
 $\beta:=\beta_{F}$. 
 On the other hand, we have
 \begin{equation}\label{eq:a2}
    0=A_{xx}|_{x=0}=2\alpha\alpha'+\alpha'',
 \end{equation}
 where the prime means the derivative with respect to $y$.
 Hence $\alpha'+\alpha^2$
 is a constant function, and get the first relation. 
 We then get 
 \begin{equation}\label{eq:a2a}
          0=A_{xxx}|_{x=0}=4 \alpha \beta'+\beta'',
 \end{equation}
 which yields the second assertion.
 The last assertion follows from the fact that
 $B_{xx}(0,y)=-2(\alpha'+\alpha^2)= 2\mu_F$.
\end{proof}

We can find the solution $\alpha=\alpha_F$ 
of the ordinary differential equation
\eqref{eq:c} under the conditions
\[
   \alpha(0)=\ddot u(0),\qquad
      \mu_F=-\bigl(\ddot u(0)^2+\ddot v(0)\bigr)
\]
for a given initial curve $\gamma(x)=(u(x),v(x))$.
By a homothetic change
\[
   \widetilde{F}(x,y):=(\tilde f(x,y),x,y)\qquad
   \left(\tilde f(x,y):= \frac{1}{m}f(mx,my),~m>0\right),
\]
one can normalize the characteristic $\mu_F$ to be $-1$, $0$ or $1$.
In fact, as shown in \cite{CR},
\[
 \alpha^+:=-\tan(y+c)
 \qquad (|c|<\pi/2)
\]
is a general solution of \eqref{eq:c} for $\mu_F=1$,
\[
  \alpha^0_{I}:=0 \quad
  \mbox{and}
  \quad \alpha^0_{\II}:= \frac{1}{y+c} 
  \qquad (c\in \R\setminus \{0\})
\]
are the solutions for $\mu_F=0$, and
\begin{alignat*}{2}
 \alpha^-_{I}&:=\tanh(y+c)
 \qquad &(c&\in\R),
 \\
 \alpha^-_{\II}&:=\coth(y+c)
 \qquad &(c&\in \R\setminus \{0\}),\\
 \alpha^-_{\III}&:=\pm 1 
\end{alignat*}
are the solutions for $\mu_F=-1$.
Thus, as pointed out in \cite{CR},
$\Y^\omega_b$ consists of the following six subclasses:
\[
  \Y^+,\quad \Y^0_{I},\quad 
  \Y^0_{\II},\quad \Y^-_{I},
  \quad \Y^-_{\II},
  \quad \Y^-_{\III}.
\]

\begin{remark}
 To find surfaces in the class
 $\Y^-_{\III}$, we may set $\alpha^-_{\III}:=1$ without loss of
 generality. In fact, if $F\in \Y^\omega_b$ satisfies $\alpha_F=-1$,
 then we can write
 \[
    F=\bigl(f(x,y),x,y\bigr),\qquad f(x,y)=y-\frac{x^2}2+k(x,y)x^3,   
 \]
 where $k(x,y)$ is a $C^\omega$-function defined on
 a neighborhood of the origin.
 If we set $X:=-x,Y:=y$ and set $\tilde F:=-F$, then
 we have
 \[
   \tilde F(X,Y)=\left(Y+\frac{X^2}2+k(-X,Y)X^3,X,Y\right).
 \]
 Thus $\alpha_{\tilde F}=1$.
\end{remark}

Like as in the case of $\alpha=\alpha_F$,
we can find the solution $\beta:=\beta_F$ of
the ordinary differential equation \eqref{eq:c2}
with the given initial condition
$2(\beta(0),\beta'(0))=\dddot \gamma_{F}(0)$.
In fact, $\beta$ can be written explicitly for each 
$\alpha=\alpha^+$, $\alpha^0_I$, $\alpha^0_{\II}$,
$\alpha^-_I$, $\alpha^-_{\II}$, $\alpha^-_{\III}$
as follows:
\begin{align*}
 \beta^+&=c_1 \left(2+\sec^2(y+c)\right)\tan(y+c)+c_2, \\
 \beta^0_I&=c_1 y+c_2, \\
 \beta^0_{\II}&=\frac{c_1}{(y + c)^3}+c_2, \\
 \beta^-_{I}&= c_1 \left(2 + 
 \sech^2(y + c)\right) \tanh(y + c)+c_2, \\
 \beta^-_{\II}&=
        c_1 \left(2-\csch^2(y+c)\right)\coth(y+c)+c_2, \\
 \beta^-_{\III}&=c_1e^{\pm 4y}+c_2.
\end{align*}
In particular, we get the following assertion:

\begin{proposition}
 The second and the third approximation functions of each 
 $F\in \Y^\omega_b$
 can be written in terms of elementary functions.
\end{proposition}

The following assertion implies that 
our approximation for $F\in \Y^\omega$
itself is an element of  $\Y^\omega$:

\begin{proposition}\label{prop:N0}
 Let $\alpha$ and  $\beta$ be analytic functions satisfying
 \begin{equation}\label{eq:alpha-beta}
   \alpha''+2\alpha\alpha'=0,\qquad \beta''+4\alpha\beta'=0.
 \end{equation}
 Then the immersion
 \[
   F_{\alpha,\beta}(x,y):=
       \left(y+\frac{\alpha(y)}{2}x^2+\frac{\beta(y)}{3}x^3,x,y\right)
 \]
 belongs to the class $\Y^\omega$. 
\end{proposition}

This assertion follows from the proof of
Proposition \ref{prop:2ode} immediately.
Moreover, the following assertion holds:

\begin{proposition}\label{thm:N1}
 We fix an analytic function $\phi\in C^\omega_o(\R^2)$.
 Let $\alpha$, $\beta$ be analytic functions satisfying
 \eqref{eq:alpha-beta}.
 Then there exists an immersion
 $F\in \Y^\omega_{\phi}$ 
 such that $\alpha=\alpha_F$ and $\beta=\alpha_F$.
 In the case of $\phi=0$,
 this implies the existence of  
 $F_{\alpha,\beta}\in \Z^\omega_{b}$.
\end{proposition}
\begin{proof}
 We set 
 \[
   \gamma(x)=(1,0)+\frac{x^2}2(\alpha(0),\beta(0))+
           \frac{x^3}3(\alpha'(0),\beta'(0)).
 \]
 By Theorem \ref{thm:E0},
 there exists a unique immersion 
 $F\in \Y^\omega_{\phi}\cap\Y^\omega_b $
 such that $\gamma_F=\gamma$.
 Then we have $\alpha_F=\alpha$ and $\beta_F=\beta$,
 proving the assertion.
\end{proof}
Finally, we prove the following assertion for the sign 
of the Gaussian curvature near a degenerate light-like point.
(cf.\ Proposition~\ref{prop:PK1} for the non-degenerate case.)
\begin{proposition}\label{prop:PK2}
 Let $F$ be an immersion in the class $\Y^\omega_b$.
 Then the Gaussian curvature function $K$ diverges 
 to $\infty$ at a degenerate light-like point if $\mu_F> 0$. 
 On the other hand, if $\mu_F=0$ and $\delta_F\ne 0$,
 then $K(x,y)$ diverges to $+\infty$ {\rm(}resp.\ $-\infty${\rm)}
 on the domain of $B_F(x,y)>0$ 
 {\rm(}resp.\ $B_F(x,y)<0${\rm)}
 as $(x,y)\to (0,0)$,
 where 
 \begin{equation}\label{eq:delta}
    \delta_F:= \beta'(0) + 3\alpha(0)\beta(0)
 \end{equation}
 and $\alpha=\alpha_F$, $\beta=\beta_F$.
\end{proposition}
\begin{proof}
 Recall that the Gaussian curvature $K$ is expressed as 
 (cf.\ \eqref{eq:BC}) 
 \[
   K=-\frac{C_F}{(B_F)^2}\qquad
   (C_F=f_{xx}f_{yy}-f_{xy}^2),
 \]
 where $f=\iota_F$ and $B_F$ is the function defined as 
 in \eqref{eq:B}.
 Since $F\in\Y^\omega_b$, can be expanded as \eqref{zm22b}.
 Then we have
 \begin{align*}
   C_F&= f_{xx}f_{yy}-(f_{xy})^2\\
    &=\left(\frac{1}{2} \alpha \alpha''-(\alpha')^2\right)x^2 
  +\left(\beta \alpha''-2 \alpha' \beta'+\frac{1}{3} \alpha \beta''\right)
  x^3+\mbox{(higher order terms)}.
 \end{align*}
 Since $F\in \Y^r_b$, the relations\eqref{eq:c} and \eqref{eq:c2}
 hold, and then we have
 \[
   C_F=\alpha'\mu_F x^2-\frac{2}{3} 
    \left(3 \alpha' \beta'+3 \alpha \beta \alpha'+2 \alpha^2 \beta'\right)
     x^3+\mbox{(higher order terms)}.
 \]
 If $\mu_F>0$, then $\alpha'(0)<0$ by \eqref{eq:c},
 so we get the conclusion.
 We next assume $\mu_F=0$, then
 $\alpha'=-\alpha^2$ and we have
 \[
   C_F(x,0)=
         \frac{2\alpha(0)^2\delta_F}{3} 
         x^3+\mbox{(higher order terms)}.
 \]
 In this situation, it holds that
 \[
      B_F(x,0)=-\frac{2\delta_F}{3}x^3+
          \mbox{(higher order terms)}.
 \]
 Thus the sign of the Gaussian curvature $K(x,0)$ 
 coincides with that of $B_F(x,0)$, proving the  
 assertion.
\end{proof}

\section{Examples}
In this section, we give several examples
of zero mean curvature surfaces:
We now give here a recipe to give more refined
approximate solutions as follows:
For $F\in\Z_b^\omega$, we can expand the function $f=\iota_F$ as
\begin{equation}\label{eq:0}
 f(x,y)=y+\sum_{k=2}^\infty \frac{a_{k}(y)}{k}x^{k}.
\end{equation}
We call  each function $a_{k}(y)$ as
the \emph{$k$-th approximation function} of $F$.
Remark that $a_2$ and $a_3$ coincide with 
$\alpha_F$ and $\beta_F$ in \eqref{zm22b}, respectively:
\begin{equation}\label{eq:alpha0}
  \alpha_F=a_2, \qquad \beta_F=a_3. 
\end{equation}

We give here several examples:
\begin{example}
 The light-like plane (cf.\ Example \ref{ex:A})
 $F(x,y)=(y,x,y)$
 belongs to $\Lambda^\omega\cap \Z^0_I$ such that
 $\gamma_F=(0,1)$
 and $\alpha_F=\beta_F=0$.
\end{example}

\begin{example}
 The light-cone (cf.\ Example \ref{ex:B})
 \[
    F(x,y)=(\sqrt{x^2+(1+y)^2}-1,x,y)
 \]
 also belongs to $\Lambda^\omega\cap \Z^0_{\II}$ such that
 \[
   \gamma_F=(\sqrt{1+x^2}-1,1/{\sqrt{1+x^2}})
 \]
 and 
 $\alpha_F=1/(1+y)$, $\beta_F=0$.
\end{example}

\begin{example}
 The surface
 $F(x,y)=(y+x^2/2,x,y)$
 is a zero-mean curvature surface 
 in $\Z^-_{\III}$,
 which satisfies
 $\gamma_F=(x^2/2,0)$
 and 
 \[
   \alpha_F=1,\qquad \beta_F=0.
 \]
\end{example}

\begin{example}
 Recall the space-like Scherk surface 
 $\{(t,x,y)\in \R^3_1\,;\,\cos t = \cos x\,cos y\}$
 (cf.\ \cite[Example 3]{CR}).
 We replace $(t,x,y)$ by $(t+(\pi/2),x,(\pi/2)-y)$,
 we have the expression
 \[
   F(x,y)=\left(-\arccos(\cos x \sin y)-\frac{\pi}{2},x,y\right)
 \]
 satisfying \eqref{eq:normalize}.
 This is a zero-mean curvature surface in $\Z^+$,
 which satisfies
 $\gamma_F=(-\pi,\cos x)$
 and 
 \[
   \alpha_F=-\tan y,\qquad \beta_F=0.
 \]
\end{example}

\begin{example}
 The time-like Scherk surface of the first kind 
 (cf.\ \cite[Example 4]{CR})
 can be normalized as 
 \[
   F(x,y)=\biggl(\arccosh(\cosh x \cosh(y+1))-1,x,y\biggr),
 \]
 which 
 is a zero-mean curvature surface in $\Z^-_{I}$,
 This satisfies
 \[
   \gamma_F=\left(-1+\arccosh( \cosh x \cosh 1),
   \frac{\sinh 1\cosh x }{\sqrt{(\cosh 1\cosh x)^2-1}}\right)
 \]
 and 
 \[
   \alpha_F=\coth y,\qquad \beta_F=0.
 \]
\end{example}

\begin{example}
 The time-like Scherk surface of the second kind 
 (cf.\ \cite[Example 5]{CR})
 \[
   F(x,y)=\biggl(\op{arcsinh}(\cosh x \sinh y),x,y\biggl)
 \]
 is a zero-mean curvature surface in $\Z^-_{\II}$,
 which satisfies
 $\gamma_F=(0,\cosh x)$
 and 
 \[
   \alpha_F=\tanh y,\qquad \beta_F=0
 \]
 hold.
\end{example}
\medskip
For $F\in\Z^\omega_b$,
it holds for $k\ge 4$ that
\begin{equation}\label{eq:a2k}
\left.\frac{d^kA}{dx^k}\right|_{x=0}=0\qquad (A:=A_F),
\end{equation}
which can be considered as an ordinary differential equation
of the $k$-th approximation function $a_k$ as in \eqref{eq:0}.
As shown in \cite{OJM}, 
\eqref{eq:a2k} is equivalent to
\begin{equation}
\label{eq:a22}
  a''_k+2(k-1)a_2a'_k+k(3-k)a'_2a_k +k(P_k+Q_k-R_k)=0,
\end{equation}
where $P_k$, $Q_k$, $R_k$ are terms written
using $\{a_s\}_{s<k}$ as follows:
\begin{align*}
 P_k&:=\sum_{m=3}^{k-1}\frac{2(k-2m+3)}{k-m+2} a_m a'_{k-m+2}, \\
 Q_k&:=
 \sum_{m=2}^{k-2}\sum_{n=2}^{k-m}
 \frac{3n-k+m-1}{mn}a'_m a'_{n}a_{k-m-n+2}, \\
 R_k&:=
 \sum_{m=2}^{k-2}\sum_{n=2}^{k-m}
 \frac{a_m a_{n}a''_{k-m-n+2}}{k-m-n+2}. 
\end{align*}
When $k=4$, \eqref{eq:a22}
reduces to
\[
   a''_4+6 a_2 a'_4
   -4 a'_2a_4 +
   3 a_2 (a'_2)^2
   -2 (a_2)^2 a''_2+
   \frac{8}{3} a_3 a'_3=0.
\]
Using \eqref{eq:a22} and \eqref{eq:0},
one can get an appropriate 
approximation  
for $F$. 

Finally, we remark on existence results
of zero mean curvature surface 
using Corollary \ref{cor:Z00}:
For $F\in \Z^\omega$, we set
\[
   \gamma_F=(0,1)+(0,v_1)x+
     \sum_{n=2}^4(u_n,v_n)\frac{x^{n}}{n}
      +\mbox{(higher order terms)}.
\]
Then $F\in \Z^\omega_b$ if and only if $v_1=0$.
Under the assumption $v_1=0$, 
the characteristic (cf.\ \eqref{eq:c}) $\mu_F$
and the constant $\delta_F$ in \eqref{eq:delta} 
satisfy
\[
 \mu_F=-(u_2^2+v_2)\qquad\text{and}\qquad \delta_F=3u_2u_3+v_3.
\]
We set
\[
 \Delta_F:=4u_3^2+8u_2u_4+v_2^2+2v_4. 
\]

\begin{proposition}
\label{prop:main}
 A surface $F\in \Z^\omega$
 belongs to $\Z^0_{I}$ {\rm(}resp.\ $\Z^0_{\II}$\rm{)}
 if $\mu_F=0$ and $u_2=0$ {\rm(}resp.\ $\mu_F=0$ and $u_2\ne 0${\rm)}. 
 Suppose that $\mu_F=0$. Then 
\begin{enumerate}
 \item 
      $F$ changes causal type if 
      $\delta_F\ne 0$, and
 \item
     $F$ has no time-like $($resp. space-like$)$ part if 
     $\delta_F=0$ and  
     $\Delta_F<0$ $($resp.~$\Delta_F>0)$.
\end{enumerate}
\end{proposition}

\begin{proof}
 The causal type of $F$
 depends on the sign of $B:=B_F$. 
 As shown in \cite{CR}, 
 $B|_{x=0}=B_{x}|_{x=0}=0$. Moreover, one can easily see that
 \begin{equation}
  B(x,0)=\mu_F x^2-\frac{2\delta_F}{3}x^3-\frac{\Delta_F}{4}x^4
   +\mbox{(higher order terms)}.
 \end{equation}
 So we get the conclusion.
\end{proof}

\begin{example}\label{ex:OJM}
 In \cite{OJM}, $F\in \Z^\omega$ satisfying
 \[
   \gamma_F(x):=(0,1+3c x^3)
 \]
 is constructed,
 which  belongs to the class $\Z^\omega_{b}$
 and changes its causal type.
 Although  the existence of this $F$ is 
 obtained by applying Corollary \ref{cor:Z00}, 
 the advantage of the method in
 \cite{OJM} is that we can get the 
 explicit approximation for $F$ at the same time. 
\end{example}

Until now, the existence of zero mean curvature surfaces
(i.e. ZMC-surfaces)
in the following three cases 
was unknown (cf.\ the footnote of \cite[Page 194]{fluid});
\begingroup
\renewcommand{\theenumi}{{\rm(\roman{enumi})}}
\renewcommand{\labelenumi}{{\rm(\roman{enumi})}}
\begin{enumerate} 
 \item\label{enum:ZMC-1} ZMC-surfaces in $\Z^0_{I}$ without space-like part, 
 \item\label{enum:ZMC-2} ZMC-surfaces in $\Z^0_{I}$ without time-like part,
 \item\label{enum:ZMC-3} ZMC-surfaces in $\Z^0_{\II}$ which changes
      causal type. 
\end{enumerate}
We can show the existence of the above remaining cases:
\endgroup
\begin{corollary}
\label{cor:main}
 There exist ZMC-immersions
 satisfying \ref{enum:ZMC-1}, 
            \ref{enum:ZMC-2} and 
            \ref{enum:ZMC-3}, respectively.
\end{corollary}

\begin{proof}
 We set $\mu_F=0$.
 If $u_2\ne 0$ and $\delta_F\ne 0$, then 
 $F\in \Z^0_{\II}$ which changes causal type 
 (i.e.\ it
 gives the case \ref{enum:ZMC-3}).
 On the other hand,
 if $u_2=\delta_F=0$ and
 $\Delta_F<0$ (resp.~$\Delta_F>0$),
 then it gives the case \ref{enum:ZMC-2} (resp.\ \ref{enum:ZMC-1}).	
\end{proof}
\appendix
\section{Division Lemma}\label{app:div}
\begin{lemma}\label{lem:div}
 Let $g$ be a $C^r$-function {\rm(}$r\geq 1${\rm)} defined on a
 convex domain $U$ of the $xy$-plane including the origin $o$,
 satisfying
 \begin{equation}\label{eq:div-ass}
    g(0,y)=\frac{\partial g}{\partial x}(0,y)
          =\frac{\partial^2 g}{\partial x^2}(0,y)
          =\dots 
          =\frac{\partial^k g}{\partial x^k}(0,y)=0\qquad 
           \bigl((0,y)\in U\bigr)
 \end{equation}
 for a non-negative integer $k<r$.
 Then there exists a 
 $C^{r-k-1}$-function $h$ defined on $U$ such that
 \begin{equation}\label{eq:ass-div}
   g(x,y)=x^{k+1}h(x,y)\qquad \bigl((x,y)\in U\bigr).
 \end{equation}
\end{lemma}
\begin{proof}
 We shall prove by an induction in $k$.
 Since
 \begin{align*}
   g(x,y)&=
          \int_0^1 \frac{d g(tx,y)}{d t}\,dt
         = \int_0^1 x g_x(tx,y) dt=x \int_0^1 g_x(tx,y)\, dt,
 \end{align*}
 the conclusion follows for $k=0$,
 by setting
 \[
        h(x,y):=\int_0^1 g_x(tx,y)\, dt.
 \]
 Assume that the statement holds for  $k-1$.
 If $g$ satisfies \eqref{eq:div-ass}, there exists
 a $C^{r-k}$-function $\phi(x,y)$ defined on $U$
 such that
 \begin{equation}\label{eq:vh}
  g(x,y)=x^k \phi(x,y)\qquad \bigl((x,y)\in U\bigr).
 \end{equation}
 Differentiating this $k$-times in $x$,
 we have
 \[
   0=\frac{\partial^k g}{\partial x^k}(0,y)
    = k! \phi(0,y)
 \]
 because of \eqref{eq:div-ass}.
 Hence, by the case $k=0$ of this lemma,
 there exists $C^{k-r-1}$-function $h(x,y)$ defined on $U$
 such that $\varphi(x,y)=x h(x,y)$.
 The function $h$ is the desired one.
\end{proof}

\begin{acknowledgements}
 The authors thank Shintaro Akamine, Udo Hertrich-Jeromin,
 Wayne Rossman and Seong-Deog Yang
 for valuable comments.
\end{acknowledgements}

\end{document}